\documentclass[11pt]{amsart}  

\usepackage{color} 

\usepackage[text={0.75\paperwidth,0.75\paperheight}]{geometry}
\usepackage[square]{natbib} 
\usepackage[leqno]{amsmath}
\usepackage{amssymb,amsthm}
\usepackage[mathscr]{euscript}
\usepackage{bbm}
\usepackage{dsfont}
\usepackage{enumerate}
\usepackage{stmaryrd}
\usepackage[all,arc]{xy}
\DeclareMathAlphabet{\mathpzc}{OT1}{pzc}{m}{it}

\swapnumbers 
\theoremstyle{plain}
\newtheorem{lemma}[subsection]{Lemma}
\newtheorem{proposition}[subsection]{Proposition}
\newtheorem{theorem}[subsection]{Theorem}
\newtheorem{corollary}[subsection]{Corollary}

\theoremstyle{definition}

\theoremstyle{remark}

\newenvironment{eqcond}{\begin{enumerate}}{\end{enumerate}}

\newcommand{\df}[1]{\emph{#1}}

\newcommand{\fa}{\mathfrak{a}}

\newcommand{\fp}{\mathfrak{p}}
\newcommand{\fq}{\mathfrak{q}}

\newcommand{\fv}{\mathfrak{v}}
\newcommand{\fw}{\mathfrak{w}}
\newcommand{\fx}{\mathfrak{x}}
\newcommand{\fy}{\mathfrak{y}}

\newcommand{\fP}{\mathfrak{P}}
\newcommand{\fQ}{\mathfrak{Q}}

\newcommand{\fX}{\mathfrak{X}}

\newcommand{\yoneda}{\mathpzc{y}\hspace*{0.5pt}} 

\DeclareMathOperator{\ev}{ev}
\DeclareMathOperator{\maxop}{\vee}
\DeclareMathOperator{\maxact}{\vee}

\newcommand{\mate}[1]{\,^\ulcorner\! #1^\urcorner}

\newcommand{\catfont}[1]{\mathsf{#1}}

\newcommand{\SET}{\catfont{Set}}

\newcommand{\TOP}{\catfont{Top}}
\newcommand{\AP}{\catfont{App}}
\newcommand{\PSAP}{\catfont{PsApp}}
\newcommand{\MET}{\catfont{Met}}
\newcommand{\BOOL}{\catfont{Bool}}

\newcommand{\Mat}[1]{#1\text{-}\catfont{Rel}}

\xyoption{tips}
\SelectTips{cm}{10}
\UseTips

\newcommand{\relto}{{\longrightarrow\hspace*{-2.8ex}{\mapstochar}\hspace*{2.6ex}}}

\newcommand{\monadfont}[1]{\mathbbm{#1}}

\newcommand{\mU}{\monadfont{U}}
\newcommand{\umonad}{(U,e,m)}

\newcommand{\eU}{\overline{U}}

\newcommand{\doo}[1]{\overset{\centerdot}{#1}}
\newcommand{\eps}{\varepsilon}
\newcommand{\op}{\mathrm{op}}

\newcommand{\true}{\mathsf{true}}
\newcommand{\false}{\mathsf{false}}

\newcommand\adjunct[2]{\xymatrix@=8ex{\ar@{}[r]|{\top}\ar@<1mm>@/^1mm/[r]^{{#2}} & \ar@<1mm>@/^1mm/[l]^{{#1}}}}

\title{Exponentiable approach spaces}
\author{Dirk Hofmann}
\address{CIDMA -- Center for Research and Development in Mathematics and Applications, Department of Mathematics, University of Aveiro, 3810-193 Aveiro, Portugal}
\email{dirk@ua.pt}
\author{Gavin J.~Seal}
\address{Mathematics Institute for Geometry and Applications, Ecole Polytechnique F\'ed\'erale de Lausanne, 1015 Lausanne, Switzerland}
\email{gavin.seal@fastmail.fm}
\thanks{Partial financial assistance by FEDER funds through COMPETE -- Operational Programme Factors of Competitiveness (Programa Operacional Factores de Competitividade) and by Portuguese funds through the Center for Research and Development in Mathematics and Applications (University of Aveiro) and the Portuguese Foundation for Science and Technology (FCT -- Funda\c{c}\~ao para a Ci\^encia e a Tecnologia), within project PEst-C/MAT/UI4106/2011 with COMPETE number FCOMP-01-0124-FEDER-022690, and the project MONDRIAN under the contract PTDC/EIA-CCO/108302/2008 with COMPETE number FCOMP-01-0124-FEDER-010047 is gratefully acknowledge.}

\date{\today} 

\subjclass[2010]{54A05, 54B30, 54C35}
\keywords{Approach space, ultrafilter convergence, exponentiable object}

\usepackage[pdftex]{hyperref}

\begin{document}

\begin{abstract}
In this note we present a characterisation of exponentiable approach spaces in terms of ultrafilter convergence.
\end{abstract}

\maketitle

\section{Introduction}

The category $\AP$ of approach spaces was introduced in \citep{Low89} as a common framework for the study of topological and metric structures. More precisley, $\AP$ contains the category $\TOP$ of topological spaces as a reflective and coreflective full subcategory, and the category $\MET$ of generalised metric spaces (see \citep{Law73}) as a coreflective full subcategory. However, just like $\TOP$ and $\MET$, $\AP$ is not cartesian closed. This fact triggers the question of finding sufficient and necessary conditions for an approach space $X$ to be exponentiable, that is, for the cartesian product functor $(-)\times X:\AP\to\AP$ to have a right adjoint. A first important result in this direction was obtained in \citep{LS04} where it is shown that every compact Hausdorff spaces is exponentiable in the category of uniform approach spaces and contractions. Two years later, \citep{Hof06} presents a sufficient condition motivated by the characterisation of exponentiable generalised metric spaces obtained in \cite{CH06} (see Theorem \ref{thm:AppExpsuff}). This condition implies in particular that
\begin{enumerate}[\ \raisebox{1pt}{\tiny $\bullet$}]
\item a topological space is exponentiable in $\TOP$ if and only if it is exponentiable in $\AP$;
\item a generalised metric space is exponentiable in $\MET$ if and only if it is exponentiable in $\AP$;
\item every injective approach space is exponentiable in $\AP$ (see \citep[Theorem~5.14]{Hof13}).
\end{enumerate}
The aim of this note is to show that this sufficient condition is also necessary.

Note that related results where obtained in \citep{LLV97} and \citep{Hof07}. In the former paper, the authors characterise exponentiable pre-approach spaces, whereby the latter considers a slightly different product $X\otimes Y$ of spaces and characterises those approach spaces $X$ for which $(-)\otimes X:\AP\to\AP$ has a right adjoint. In this paper we follow closely the proof of \citep[Theorem~6.9]{Hof07}.


\section{Approach spaces}\label{ssec:App}

Approach spaces were introduced in \citep{Low89} and are comprehensively described in \citep{Low97}. If not stated otherwise, for notation and results we refer to \citep{Low97}.

By definition, an \df{approach space} is a set $X$ together with a function $\delta:PX\times X\to[0,\infty]$ (called a \df{distance function} or an \df{approach distance}) subject to
\begin{enumerate}
\item $\delta(\{x\},x)=0$,
\item $\delta(\varnothing,x)=\infty$,
\item $\delta(A\cup B,x)=\min\{\delta(A,x),\delta(B,x)\}$,
\item $\delta(A^{(\eps)},x)+\eps\geqslant\delta(A,x)$ (where $A^{(\eps)}=\{x\in X\mid \eps\geqslant\delta(A,x)\}$);
\end{enumerate} 
for all $A,B\subseteq X$, $x\in X$ and $\eps\in[0,\infty]$. For approach spaces $X$ and $Y$ with distance functions $\delta:PX\times X\to[0,\infty]$ and $\delta':PY\times Y\to[0,\infty]$ respectively, a map $f:X\to Y$ is a \df{contraction} if $\delta(A,x)\geqslant\delta'(f(A),f(x))$ for all $A\subseteq X$ and $x\in X$. Approach spaces and contraction maps are the objects and morphisms of the category $\AP$.

The forgetful functor
\[
 \AP\to\SET
\]
is topological, and therefore $\AP$ is complete and cocomplete and $\AP\to\SET$ preserves limits and colimits. Furthermore, the functor $\AP\to\SET$ factors through the canonical forgetful functor $\TOP\to\SET$, where $\AP\to\TOP$ sends an approach space $(X,\delta)$ to the topological space $X$ with
\[
x\in\overline{A} \iff\delta(A,x)=0.
\]
Moreover, $\AP\to\TOP$ has a fully faithful left adjoint $\TOP\to\AP$ that interprets a topological space $X$ as the approach space $X$ with distance function
\[
 \delta(A,x)=
 \begin{cases}
  0 &\text{if $x\in\overline{A}$},\\
  \infty &\text{otherwise.}
 \end{cases}
\]
Via the fully faithful functor $\TOP\to\AP$, we can consider every topological space as an approach space. Being left adjoint, $\TOP\to\AP$ preserves all colimits, but $\TOP\to\AP$ also preserves all limits and therefore also has a left adjoint.

As is the case for topological spaces, the structure of an approach space can be described in several equivalent ways, the most relevant for this paper is by ultrafilter convergence. Explicitly, let $UX$ denote the set of all ultrafilters on the set $X$; then each function $\delta:PX\times X\to[0,\infty]$ defines a map $a:UX\times X\to[0,\infty]$ via
\[a(\fx,x)=\sup_{A\in\fx}\delta(A,x),\]
and vice versa, each $a:UX\times X\to[0,\infty]$ defines a function $\delta:PX\times X\to[0,\infty]$ via
\[\delta(A,x)=\inf_{A\in\fx}a(\fx,x);\]
moreover, every approach distance is completely determined by its corresponding ultrafilter convergence. Axioms characterising those numerical relations $UX\relto X$ induced this way by an approach distance are given in \citep{LL88}; however, we will use here the description obtained in \citep{CH03} that we recall next.

\section{The ultrafilter monad on numerical relations}

We start by describing \df{numerical relations} that should be seen as relations with truth values in $[0,\infty]$. Here, $0$ corresponds to $\true$ and $\infty$ to $\false$, and we consider $[0,\infty]$ with its natural order. With respect to this order, the addition $u+(-):[0,\infty]\to[0,\infty]$ and truncated subtraction $(-)\ominus u:[0,\infty]\to[0,\infty]$ (given by $v\ominus u=\max\{v-u,0\}$) are adjoint, so that for all $u,v,w\in [0,\infty]$
\[
 u+v\geqslant w\iff v \geqslant w\ominus u.
\]

A \df{numerical relation} $r:X\relto Y$ from a set $X$ to a set $Y$ is a map $r:X\times Y\to[0,\infty]$. The composite $s\cdot r:X\relto Z$ of $r:X\relto Y$ with $s:Y\relto Z$ is the numerical relation defined by
\begin{align*}
s\cdot r(x,z)&=\inf_{y\in Y}(r(x,y)+s(y,z))
\end{align*}
for all $x\in X$, $y\in Y$, $z\in Z$. Every ordinary relation becomes a numerical relation by interpreting $\true$ as $0$ and $\false$ as $\infty$, and with this interpretation the identity function is also the identity numerical relation. Numerical relations are also ordered via
\[
r\geqslant r'\iff\text{for all $x\in X$ and $y\in Y$, }r(x,y)\geqslant r'(x,y),
\]
(for $r,r':X\relto Y$), and since composition of numerical relations preserves this order in both variables, sets with numerical relations form an ordered category
\[
\Mat{[0,\infty]}.
\]
Furthermore, for $r:X\relto Y$ in $\Mat{[0,\infty]}$, we define $r^\circ:Y\relto X$ by $r^\circ(y,x)=r(x,y)$ (for all $x\in X$, $y\in Y$), and obtain this way a locally monotone functor $(-)^\circ:\Mat{[0,\infty]}^\op\to\Mat{[0,\infty]}$.

Given $u,v\in[0,\infty]$, we denote the maximum of $u$ and $v$ by $u\maxop v$. More generally, every $u\in[0,\infty]$ induces an action on $\Mat{[0,\infty]}(X,Y)$: for $r:X\relto Y$, we define $u\maxact r:X\relto Y$ by $(u\maxact r)(x,y)=u\maxop r(x,y)$, for all $x\in X$ and $y\in Y$.

\begin{lemma}\label{lem:SomeRulesMax}
Let $u,v\in[0,\infty]$, $r:X\relto Y$, $s:Y\relto Z$ in $\Mat{[0,\infty]}$ and $f:X\to Y$ be a map. Then the following hold:
\begin{enumerate}
\item $(v\maxact s)\cdot(u\maxact r)\geqslant (v+u)\maxact (s\cdot r)$,
\item $(v\maxact s)\cdot f=v\maxact (s\cdot f)$.
\end{enumerate}
\end{lemma}

The ultrafilter monad $\mU=\umonad$ on $\SET$ is induced by the adjunction
\[
 \BOOL^\op\adjunct{\hom(-,2)}{\hom(-,2)}\SET,
\]
where $\BOOL$ denotes the category of Boolean algebras and homomorphisms. Explicitly, the ultrafilter functor $U:\SET\to\SET$ sends a set $X$ to the set $UX$ of all ultrafilters on $X$, and for a map $f:X\to Y$, the map  $Uf:UX\to UY$ sends $\fx\in UX$ to the ultrafilter $Uf(\fx)=\{B\subseteq Y\mid f^{-1}(B)\in\fx\}$. The natural transformations $e:1\to U$ and $m:UU\to U$ have as components at $X$ the maps
\begin{alignat*}{3}
 e_X:X & \to UX &\qquad\text{and}\qquad&& m_X: UUX &\to UX\\
x & \mapsto\doo{x}=\{A\subseteq X\mid x\in A\} &&&
\fX & \mapsto\{A\subseteq X\mid A^\sharp\in\fX\}
\end{alignat*}
respectively, where $A^\sharp=\{\fa\in UX\mid A\in\fa\}$. We also mention that the functor $U:\SET\to\SET$ preserves weak pullbacks.

The Eilenberg--Moore algebras for the ultrafilter monad $\mU=\umonad$ on $\SET$ are identified in \citep{Man69} as precisely the compact Hausdorff spaces with ultrafilter convergence as structure, and the $\mU$-homo\-morphisms are the continuous maps. A central example of a compact Hausdorff space is the free $\mU$-algebra $UX$ with ultrafilter convergence $m_X:UUX\to UX$. Recall from \ref{ssec:App} that $UX$ can be also viewed as an approach space via the embedding $\TOP\to\AP$. We also frequently use the compact Hausdorff space $[0,\infty]$ with convergence (that is, its $\mU$-algebra structure)
\[
 \xi:U[0,\infty]\to[0,\infty],\,\fv\mapsto\sup_{A\in\fv}\inf_{u\in A}u.
\]

The ultrafilter functor $U:\SET\to\SET$ extends to a locally monotone functor $\eU:\Mat{[0,\infty]}\to\Mat{[0,\infty]}$ defined by
\[
\eU r(\fx,\fy)=\sup_{A\in\fx,B\in\fy}\inf_{x\in A,y\in B}r(x,y),
\]
for a numerical relation $r:X\times Y\to[0,\infty]$ and $\fx\in UX$, $\fy\in UY$. The following alternative description of $\eU r$ will be useful in the sequel (see \citep{Hof07} and \citep[Subsection 4.1]{CH09}).

\begin{proposition}\label{prop:extU}
For every $r:X\relto Y$ in $\Mat{[0,\infty]}$, $\fx\in UX$ and $\fy\in UY$,
\[
\eU r(\fx,\fy)=\inf\{\xi\cdot Ur(\fw)\mid \fw\in U(X\times Y),\,U\pi_1(\fw)=\fx,\,U\pi_2(\fw)=\fy\}.
\]
\end{proposition}

We also note that $\eU(r^\circ)=(\eU r)^\circ$ for all numerical relations $r$; the multiplication $m$ remains a natural transformation $m:\eU\,\eU\to\eU$, but in general $e:1\to \eU$ satisfies only $e_Y\cdot r\geqslant \eU r\cdot e_X$ for numerical relations $r:X\relto Y$. Finally, for all $u\in[0,\infty]$ and $r:X\relto Y$ in $\Mat{[0,\infty]}$, we have $\eU(u\maxact r)=u\maxact \eU r$.

\section{Approach spaces via convergence}

We now have all necessary ingredients to present the characterisation of the ultrafilter convergence relation of an approach space obtained in \citep{CH03}: a numerical relation $a:UX\relto X$ is induced by an approach distance function $\delta:PX\times X\to[0,\infty]$ if and only if
\begin{alignat*}{3}
 e_X^\circ\geqslant a &\qquad\text{and}\qquad& a\cdot \eU a\geqslant a\cdot m_X.
\end{alignat*}
A numerical relation $a:UX\relto X$ satisfying the first inequality is called \df{reflexive}, and $a$ is called \df{transitive} if it satisfies the second inequality. Pointwise, these formulas read as
\begin{alignat*}{3}
 0\geqslant a(\doo{x},x) &\qquad\text{and}\qquad& \eU a(\fX,\fx) + a(\fx,x)\geqslant a(m_X(\fX),x),
\end{alignat*}
for all $\fX\in UUX$, $\fx\in UX$ and $x\in X$. For approach spaces $X$ and $Y$ with ultrafilter convergence $a:UX\relto X$ and $b:UY\relto Y$ respectively, a map $f:X\to Y$ is a contraction map if and only if
\[
f\cdot a\geqslant b\cdot Uf,
\]
which is equivalent to $a\geqslant f^\circ\cdot b\cdot Uf $, and reads in pointwise notation as
\[
a(\fx,x)\geqslant b(Uf(\fx),f(x)),
\]
for all $\fx\in UX$ and $x\in X$. Since the condition $f\cdot a\geqslant b\cdot Uf$ does not refer to any property of $a$ or $b$, we will use the terminology ``contraction'' also in contexts where $a:UX\relto X$ and $b:UY\relto Y$ are just numerical relations. For instance, functoriality of $\eU$ implies at once that if $f$ is a contraction, then so is $Uf:UX\relto UY$, where we consider the convergence relations $\eU a:UUX\relto UX$ and $\eU b:UUY\relto UY$ on $UX$ and $UY$, respectively; but $\eU a$ is in general neither reflexive nor transitive. 

The extended real half-line $[0,\infty]$ becomes an approach space with convergence
\[
 U[0,\infty]\times[0,\infty]\to[0,\infty],\,(\fv,v)\mapsto v\ominus\xi(\fv).
\]
The approach space $[0,\infty]$ takes the role of the Sierpi\'nski space, in particular, $[0,\infty]$ is initially dense in $\AP$. For more information we refer to \citep[Example 1.8.33 and Proposition 1.10.8]{Low97}. We now present some results that relate $[0,\infty]$ with contractions, and that we will use in our main result.

\begin{lemma}\label{lem:SomeMaps}
The following assertions hold.
\begin{enumerate}
\item The binary suprema map $\maxop:[0,\infty]\times[0,\infty]\to[0,\infty]$ is a contraction.
\item For every $u\in[0,\infty]$, the map $t_u:[0,\infty]\to[0,\infty],\,v\mapsto u\maxop v$ is a contraction.
\item For each approach space $X$, the convergence $a:UX\times X\to[0,\infty]$ is a contraction.
\end{enumerate}
\end{lemma}
\begin{proof}
The first two assertions are immediate. The third is essentially \citep[Lemma~6.7]{Hof07}.
\end{proof}

One can equivalently consider a map $\varphi:X\to[0,\infty]$ as a numerical relation $\varphi:1\relto X$, and with this interpretation one has:

\begin{proposition}
Let $X$ be an approach space with convergence $a:UX\relto X$ and let $\varphi:X\to[0,\infty]$ be a map. Then $\varphi:X\to[0,\infty]$ is a contraction if and only if the numerical relation $\varphi:1\relto X$ satisfies $a\cdot \eU\varphi\cdot e_1\geqslant\varphi$.
\end{proposition}
\begin{proof}
See \citep[Theorem 4.3]{CH09}.
\end{proof}

The map $\varphi_{u,v}$ of the following corollary will be used in the proof of Theorem \ref{thm:AppExpChar}.

\begin{corollary}\label{cor:phi_u_v_cont}
Let $X$ be an approach space, $u,v\in[0,\infty]$ and $\fX\in UUX$. Then
\[
 \varphi_{u,v}:X\to[0,\infty],\, x\mapsto\inf\{(u\maxop a(\fx,x))+(v\maxop\eU a(\fX,\fx))\mid \fx\in UX\}
\]
is a contraction.
\end{corollary}
\begin{proof}
Let $i_\fX:1\to UUX$ be the map that points to $\fX\in UUX$. Then the numerical relation $\varphi=\varphi_{u,v}:1\relto X$ is the composite
\[
 \xymatrix{1\ar[r]|-{\object@{|}}^-{\rule[-3pt]{0pt}{0pt}i_\fX} & UUX\ar[r]|-{\object@{|}}^-{\rule[-2pt]{0pt}{0pt}v\maxop\eU a} & UX\ar[r]|-{\object@{|}}^-{\rule[-2pt]{0pt}{0pt}u\maxop a} & X}
\]
in $\Mat{[0,\infty]}$, and we calculate
\begin{align*}
 a\cdot \eU\varphi\cdot e_1
 &=(0\maxop a)\cdot(u\maxop\eU a)\cdot\eU\,\eU(v\maxop a)\cdot Ui_\fX\cdot e_1\\
 &\geqslant (u\maxop a\cdot\eU a)\cdot\eU\,\eU(v\maxop a)\cdot e_{UUX}\cdot i_\fX &&\text{(Lemma \ref{lem:SomeRulesMax})}\\
 &\geqslant (u\maxop a)\cdot m_X\cdot\eU\,\eU(v\maxop a)\cdot e_{UUX}\cdot i_\fX &&\text{($a$ is transitive and Lemma \ref{lem:SomeRulesMax})}\\
 &=(u\maxop a)\cdot \eU(v\maxop a)\cdot m_{UX}\cdot e_{UUX}\cdot i_\fX\\
 &=(u\maxop a)\cdot (v\maxop\eU a)\cdot i_\fX=\varphi. &&\qedhere
\end{align*}
\end{proof}

The convergence $c$ of the product $X\times Y$ of approach spaces $X$ and $Y$ is given by
\begin{equation}\label{eq:prodAP}
 c(\fw,(x,y))=a(U\pi_1(\fw),x))\maxop b(U\pi_2(\fw),y),
\end{equation}
for all $\fw\in U(X\times Y)$, $x\in X$ and $y\in Y$. The sufficient condition obtained in \citep{Hof06} for an approach space to be exponentiable is the following.

\begin{theorem}\label{thm:AppExpsuff}
Let $X$ be an approach space with ultrafilter convergence $a:UX\relto X$. Then $X$ is exponentiable if
\[
 (u+v)\maxop a(m_X(\fX),x_0)\geqslant\inf\{(u\maxop a(\fx,x_0))+(v\maxop\eU a(\fX,\fx))\mid \fx\in UX\},
\]
for all $\fX\in UUX$, $x_0\in X$ and $u,v\in[0,\infty]$.
\end{theorem}

\section{The cartesian closed category of pseudo-approach spaces}

Similarly to $\TOP$, the category $\AP$ is not cartesian closed since, for instance, a non-exponentiable topological space cannot be exponentiable in $\AP$. This deficiency of $\AP$ led to the introduction of cartesian closed extensions of $\AP$, one of which is the category of pseudo-approach spaces introduced in \citep{LL89}. In our setting, a \df{pseudo-approach space} is a set $X$ equipped with a numerical relation $a:UX\relto X$ that is only required to be reflexive; pseudo-approach spaces with contractions form the category $\PSAP$. The canonical forgetful functor $\PSAP\to\SET$ is topological, therefore $\PSAP$ has, and $\PSAP\to\SET$ preserves all limits and colimits. The convergence of the product $X\times Y$ of pseudo-approach spaces $X$ and $Y$ with convergence relations $a:UX\relto X$ and $b:UY\relto Y$ respectively can be calculated as in \eqref{eq:prodAP}. Moreover, the canonical inclusion functor $\AP\to\PSAP$ has a left adjoint, and $\AP$ is finally dense in $\PSAP$.

As indicated above, one of the main results of \citep{LL89} is that the category $\PSAP$ is cartesian closed, that is, the functor $(-)\times X:\PSAP\to\PSAP$ has a right adjoint $(-)^X:\PSAP\to\PSAP$, for every pseudo-approach space $X$. For pseudo-approach spaces $X$ and $Y$ with convergence relations $a:UX\relto X$ and $b:UY\relto Y$ respectively, the exponential $Y^X$ is the pseudo-approach space 
\[
 Y^X=\{\text{contractions } \varphi:X\to Y\}
\]
equipped with the ``best convergence'' $d$ making the evaluation map
\[
\ev:Y^X\times X\to Y,\,(\varphi,x)\mapsto \varphi(x)
\]
a contraction, that is:
\begin{equation}\label{FunSpStr}
d(\fp,\varphi)
=\inf\{u\in[0,\infty]\mid \forall\fq\in U\pi_1^{-1}(\fp),x\in X\,.\,u\maxop a(U\pi_2(\fq),x)\geqslant b(U\!\ev(\fq),\varphi(x))\},
\end{equation}
for all $\fp\in U(Y^X)$ and $\varphi\in Y^X$.

The link between exponentiability in $\AP$ and $\PSAP$ is exposed by the following result which is an instance of \citep[Theorem 3.3]{Sch84}.
\begin{proposition}
Let $X$ be an approach space. Then the following assertions are equivalent.
\begin{eqcond}
\item $X$ is exponentiable in $\AP$.
\item For every approach space $Y$, the pseudo-approach space $Y^X$ is actually an approach space.
\item The pseudo-approach space $[0,\infty]^X$ is an approach space.
\end{eqcond}
\end{proposition}

Therefore, it is sufficient to consider the exponential $[0,\infty]^X$ with its ultrafilter convergence $d:U([0,\infty]^X)\relto [0,\infty]^X$. In this case, the condition
\[
 u\maxop a(U\pi_2(\fq),x)\geqslant b(U\!\ev(\fq),\varphi(x))
\]
in \eqref{FunSpStr} becomes
\[
 u\maxop a(U\pi_2(\fq),x)\geqslant \varphi(x)\ominus\xi(U\!\ev(\fq)),
\]
which is equivalent to
\[
 (u\maxop a(U\pi_2(\fq),x))+\xi(U\!\ev(\fq)) \geqslant \varphi(x).
\]
Recall from Lemma \ref{lem:SomeMaps} that the map $t_u:[0,\infty]\to[0,\infty]$ is a contraction for every $u\in[0,\infty]$; the right adjoint $(-)^X$ thus yields a contraction $t_u^X:[0,\infty]^X\to[0,\infty]^X$ for every pseudo-approach space $X$. In the sequel we write $u\maxact\varphi$ instead of $t_u(\varphi)$, $u\maxact\fp$ instead of $Ut_u(\fp)$, and so on.

The following technical lemma will be used in the proof of our main result.

\begin{lemma}\label{lem:UmaxactD}
Let $X$ be a pseudo-approach space, $\fP\in UU([0,\infty]^X)$, $\fp\in U([0,\infty]^X)$, $\varphi\in [0,\infty]^X$ and $u\in[0,\infty]$. Then
\begin{align*}
 u\maxact d(\fp,\varphi)\geqslant d(\fp,u\maxact\varphi) &&\text{and}&& u\maxact \eU d(\fP,\fp)\geqslant\eU d(\fP,u\maxact\fp).
\end{align*}
\end{lemma}
\begin{proof}
To see the first inequality, let $\fq\in U([0,\infty]^X\times X)$ with $U\pi_1(\fq)=\fp$ and let $x\in X$. By definition of $d$ and the previous discussion, it is sufficient to show that
\[
 (u\maxop d(\fp,\varphi)\maxop a(U\pi_2(\fq),x))+\xi\cdot U\!\ev(\fq)\geqslant u\maxop\varphi(x).
\]
But the left-hand side above is larger or equal to
\[
 u\maxop ((d(\fp,\varphi)\maxop a(U\pi_2(\fq),x))+\xi\cdot U\!\ev(\fq)),
\]
so the assertion follows from
\[
 (d(\fp,\varphi)\maxop a(U\pi_2(\fq),x))+\xi\cdot U\!\ev(\fq)\geqslant\varphi(x).
\]
To see the second inequality, just note that in point-free notation the first one reads as
\[
 u\maxact d\geqslant t_u^\circ\cdot d,
\]
hence $u\maxact\eU d=\eU(u\maxact d)\geqslant (Ut_u)^\circ\cdot\eU d$.
\end{proof}

\section{Exponentiable approach spaces}

We are now in position to prove our main result.

\begin{theorem}\label{thm:AppExpChar}
Let $X$ be an approach space with ultrafilter convergence $a:UX\relto X$. Then $X$ is exponentiable if and only if
\[
 (u+v)\maxop a(m_X(\fX),x_0)\geqslant\inf\{(u\maxop a(\fx,x_0))+(v\maxop\eU a(\fX,\fx))\mid \fx\in UX\}.
\]
for all $\fX\in UUX$, $x_0\in X$ and $u,v\in[0,\infty]$.
\end{theorem}
\begin{proof}
By Theorem~\ref{thm:AppExpsuff}, we only need to show that the condition is necessary for $X$ to be exponentiable. To this end, assume that $X$ is an exponentiable approach space and let $\fX\in UUX$ and $x_0\in X$. Set $\yoneda=\mate{a}:UX\to[0,\infty]^X$, $\yoneda_0=\yoneda\cdot e_X:X\to[0,\infty]^X$ and
\begin{alignat*}{4}
 \fp=U\yoneda(\fX), &\qquad \fP=UU\yoneda_0(\fX)&\qquad\text{and}\qquad& \fQ=UU\langle\,\yoneda_0,1_X\rangle(\fX).
\end{alignat*}
Note that $UU\pi_1(\fQ)=\fP$ and $UU\pi_2(\fQ)=\fX$. As before, the ultrafilter convergence on $[0,\infty]^X$ is denoted by $d$, and in the sequel $d'$ denotes the convergence on the product space $[0,\infty]^X\times X$.

We start by showing the following facts.
\begin{enumerate}
\item\label{fact1} $0=\xi\cdot U\!\ev\cdot m_{[0,\infty]^X\times X}(\fQ)$ and $0=\eU d(\fP,\fp)$. 
\item\label{fact2} $v\geqslant \eU d(\fP,v\maxact\fp)$.
\item\label{fact3} $u\geqslant d(v\maxact\fp,\varphi_{u,v})$.
\end{enumerate}
The two equalities in \eqref{fact1} can be shown exactly as in the proof of \citep[Theorem~6.9]{Hof07}. From Lemma \ref{lem:UmaxactD} we can then infer $v=v\maxop\eU d(\fP,\fp)\geqslant\eU d(\fP,v\maxop\fp)$, which proves \eqref{fact2}. To prove \eqref{fact3}, recall first from Corollary \ref{cor:phi_u_v_cont} that $\varphi_{u,v}$ is indeed an element of the function space $[0,\infty]^X$. Let $x\in X$ and $\fq\in U([0,\infty]^X\times X)$ with $U\pi_1(\fq)=v\maxact\fp$; by definition of $d$, it suffices to verify
\[
 (u\maxop a(U\pi_2(\fq),x))+\xi\cdot U\!\ev(\fq)\geqslant \varphi_{u,v}(x).
\]
Both squares in
\[
 \xymatrix{UX\times X\ar[rr]^{\yoneda\times 1_X}\ar[d]_{\pi_1} && [0,\infty]^X\times X\ar[rr]^{t_v^X\times 1_X}\ar[d]_{\pi_1} && [0,\infty]^X\times X\ar[d]^{\pi_1}\\
  UX\ar[rr]^\yoneda &&  [0,\infty]^X\ar[rr]^{t_v^X} && [0,\infty]^X}
\]
are pullbacks, hence, since $U$ preserves weak pullbacks, there exists some $\fw\in U(UX\times X)$ with $U((t_v^X\cdot\yoneda)\times 1_X)(\fw)=\fq$ and $U\pi_1(\fw)=\fX$. By definition of $\yoneda$, we have $\ev\cdot(\yoneda\times 1_X)=a$, and moreover the diagram
\[
 \xymatrix{[0,\infty]^X\times X\ar[rr]^{t_v^X\times 1_X}\ar[d]_\ev && [0,\infty]^X\times X\ar[d]^\ev\\
    [0,\infty]\ar[rr]_{t_v} && [0,\infty]}
\]
commutes by naturality of $\ev$; therefore, by Proposition~\ref{prop:extU}, we have $\xi\cdot U\!\ev(\fq)\geqslant v\maxop \eU a(\fX,U\pi_2(\fq))$, and we can verify~\eqref{fact3}:
\[
 (u\maxop a(U\pi_2(\fq),x))+\xi\cdot U\!\ev(\fq)\geqslant (u\maxop a(U\pi_2(\fq),x))+(v\maxop \eU a(\fX,U\pi_2(\fq)))\geqslant\varphi_{u,v}(x).
\]
Hence, by transitivity of $d$, it follows from \eqref{fact2} and \eqref{fact3} that
\[
 u+v\geqslant d(m_{[0,\infty]^X}(\fP),\varphi_{u,v})
\]
and therefore, by definition of the product structure on $[0,\infty]^X\times X$, the fact that $\ev$ is a contraction, and  \eqref{fact1},
\begin{align*}
 (u+v)\maxop a(m_X(\fX),x_0)
 &\geqslant d(m_{[0,\infty]^X}(\fP),\varphi_{u,v})\maxop a(m_X(\fX),x_0)\\
 &= d'(m_{[0,\infty]^X\times X}(\fQ),(\varphi_{u,v},x_0))\\
 &\geqslant \varphi_{u,v}(x_0)\ominus \xi\cdot U\!\ev\cdot m_{[0,\infty]^X\times X}(\fQ)\\
 &=\varphi_{u,v}(x_0)=\inf\{(u\maxop a(\fx,x_0))+(v\maxop\eU a(\fX,\fx))\mid \fx\in UX\}.\qedhere
\end{align*}
\end{proof}


\end{document}